\documentclass[12pt,letterpaper,oneside,reqno]{amsart}

\numberwithin{equation}{section}

\usepackage{hyperref}

\usepackage{geometry}
\usepackage{stmaryrd}
\usepackage{mathtools}

\usepackage{amsthm,amssymb}

\theoremstyle{plain}
\newtheorem{theorem}{Theorem}

\newtheorem{lemma}[theorem]{Lemma}
\newtheorem{corollary}[theorem]{Corollary}

\newtheorem*{vN}{Theorem}

\theoremstyle{definition}

\def\appmod{\mathrel|\joinrel\approx}

\DeclareMathOperator{\AV}{AV}
\DeclareMathOperator{\card}{card}

\DeclareMathOperator{\density}{density}
\DeclareMathOperator{\ev}{ev}

\DeclareMathOperator{\img}{img}

\DeclareMathOperator{\Neg}{neg}

\DeclareMathOperator{\tp}{tp}
\DeclareMathOperator{\Ulim}{\mathcal{U}lim}

\newcommand{\cA}{\mathcal{A}}

\newcommand{\tA}{\widetilde{\cA}}
\newcommand{\cB}{\mathcal{B}}
\newcommand{\tB}{\widetilde{\cB}}

\newcommand{\Folner}{F{\o}lner}
\newcommand{\G}{\mathcal{G}}

\newcommand{\cH}{\mathcal{H}}
\newcommand{\gH}{\lsup{g}{\mathcal{H}}}
\newcommand{\tH}{\widetilde{\mathcal{H}}}
\newcommand{\tHn}{\tH^{\mathrm{n}}}
\newcommand{\tHr}{\tH^{\mathrm{r}}}
\newcommand{\tHs}{\tH^{\mathrm{s}}}
\newcommand{\tHsf}{\tHs_{\mathrm{f}}}

\newcommand{\GH}{\lsup{G}{\cH}}
\newcommand{\GtH}{\lsup{G}{\tH}}

\newcommand{\cS}{\mathcal{S}}
\newcommand{\tS}{\widetilde{\mathcal{S}}}

\newcommand{\Ltwon}[1]{\left\|#1\right\|}

\newcommand{\NN}{\mathbb{N}}

\newcommand{\R}{\mathbb{R}}
\newcommand{\bS}{\mathbf{S}}
\newcommand{\fS}{\mathfrak{S}}

\newcommand{\fT}{\mathfrak{T}}
\newcommand{\tT}{\widetilde{T}}
\newcommand{\U}{\mathcal{U}}

\newcommand{\lsup}[2]{\prescript{#1}{}{#2}}

\newcommand{\xr}{x_{\mathrm{r}}}
\newcommand{\xs}{x_{\mathrm{s}}}

\begin{document}
\title[Model Theory and the mean ergodic theorem]
{Model Theory and the mean ergodic theorem \\ for abelian unitary actions}
\author{Eduardo Due\~nez \and Jos\'e Iovino}
\address{The University of Texas at San Antonio\\
  San Antonio, TX 78249.} \email{\texttt{eduenez@utsa.edu}, \texttt{jose.iovino@utsa.edu}}
  
  \begin{abstract}
  We explore connections between von Neumann's mean ergodic theorem and concepts of model theory. 
As an application we present a proof Wiener's generalization of von Neumann's result in which the group acting on the Hilbert space~$\cH$ is any abelian group of unitary transformations of~$\cH$.
  \end{abstract}
  

\maketitle

\section{Introduction}
\label{sec:introduction}
The discrete-time version of von Neumann's mean ergodic theorem~\cite{Neumann1932} states:
\begin{vN}
  If $T$ is a unitary operator on a Hilbert space~$\cH$,
  and~$x\in\cH$, then the sequence $\{\AV_n(x)\}$ of ``ergodic
  averages''
  \begin{equation*}
    \AV_n(x) = \frac{1}{n}\sum_{k=1}^n T^n(x)
  \end{equation*}
  converges in norm, as $n\to\infty$, to the orthogonal
  projection~$\pi(x)$ of~$x$ on the (closed) subspace
  \begin{equation*}
    \lsup{T}\cH = \{ x\in\cH \mid T(x) = x \}
  \end{equation*}
  consisting of $T$-fixed elements.
\end{vN}

Many different proofs of this important result and its generalizations appeared in the decades after. 
Some of the proofs, such as Riesz's~\cite{Riesz1941}, are short and fairly elementary, and cannot easily be surpassed in those qualities.
More recently, Tao~\cite{Tao2012} outlined a short proof relying on nonstandard analysis that has at least two virtues. 
First, Tao's argument follows a very natural route, at least within the realm of nonstandard analytical ideas \emph{\`a la} Robinson~\cite{Robinson1974}. 
Second, Tao's argument can be adapted to prove the existence of ergodic averages of generalized (``polynomial'') group actions, eventually leading to a fairly natural nonstandard proof of Walsh's theorem~\cite{Walsh2012} (for which Walsh relied on standard methods).
(Tao himself outlines the nonstandard argument needed to achieve this.)

In this paper, we explore concepts from model theory underlying Tao's argument and, as an application, we prove Wiener's generalization~\cite{Wiener1939} of von Neumann's result in which the group acting on~$\cH$ is an abelian group of unitary transformations of~$\cH$ (rather than merely the cyclic group generated by a single transformation~$T$).
Ergodic averages are defined relative to a specific \Folner\ sequence for the group.
For the definition of \Folner\ sequence and of the ergodic averages, and a precise statement of this theorem, the reader is referred to section~\ref{sec:set-up} below. 
(We note that Wiener's result did not directly used \Folner\ sequences. 
Wiener assumed the group to be finitely generated, and a fixed set of generators was used to define the sequence of ergodic averages.)

Our approach is philosophically very close to Tao's, but instead of classical nonstandard analysis we use the more modern approach of model theory of Banach space structures and of types over Banach spaces originally developed by Henson and the second author~\cite{Iovino:1999a, Iovino:1999b,Henson-Iovino:2002,Iovino2014}. 
On the other hand, this manuscript is essentially self-contained; in particular, no prior knowledge of model theory is assumed of the reader. 
The only prerequisite in logic is familiarity with the concepts of structure, formula, and satisfaction ($\models$).

An advantage of the model-theoretic approach taken here is the fact that we do not need to deal with nonstandard extensions of $\mathbb{R}$, or with internal/external sets. 
Our structures are based on genuine Banach spaces and $C^*$-algebras.

 In order to study Banach spaces from a perspective of model theory one may use either Henson's logic of approximations (see \cite{Henson-Iovino:2002}), or the equivalent framework of first-order continuous logic developed by Ben Yaacov and Usvyatsov (see~\cite{Ben-Yaacov-Berenstein-Henson-Usvyatsov:2008,Ben-Yaacov-Usvyatsov:2010}). 
We have chosen the former, as we feel that it is more adequate for the context at hand.

The crucial tool in the paper is the concept of type over a Banach space structure. 
Types over discrete structures have been central in model theory since the early days, but types over Banach spaces were introduced much later, by Krivine, in his epochal paper on finite representability of $\ell_p$ in Banach lattices~\cite{Krivine:1976}. 
Krivine's types played a preeminent role in the Krivine-Maurey proof that every Banach stable space contains some $\ell_p$ almost isometrically~\cite{Krivine-Maurey:1981}.

The power of types in analysis lies in the fact that they allow one to view functional limits taken over some structure (e.g., the limit that appears in the statement of the mean ergodic theorem) as elements of some special extension of the structure. 
By results proved by Henson and the second author~\cite{Henson-Iovino:2002} (following Shelah~\cite{Shelah:1971}), the special extension, let's call it $\mathfrak C$, can be taken in such a way that there is a univocal correspondence between types and automorphisms of $\mathfrak{C}$. 
This yields a Galois theory for types, and hence for functional limits. 
See section~\ref{sec:extensions-compactness} for the details.

From the point of view of logic, the types introduced by Krivine in Banach space theory correspond to types of quantifier-free formulas. 
Full types were introduced later by the second author~\cite{Iovino:1999a, Iovino:1999b} relying on Henson's aforementioned logical formalism. 
For most applications of model theoretic ideas to Banach space structures in the current literature, quantifier-free types suffice. 
 Here we make use of the quantifier information carried by full types.

Sections~\ref{sec:language}--\ref{sec:extensions-compactness} amount to a crash course on the aspects of the theory of types over Banach structures that are needed for our purposes. 
Readers already familiar with Banach space model theory may skip these sections and jump from section~\ref{sec:set-up}, where the main players are introduced,  to the proof of the theorem, which is given in section~\ref{sec:proof}.

\section{Statement of the theorem}
\label{sec:set-up}

\subsection{Notation}
\label{sec:notation}

Throughout the paper, $G$ shall denote a discrete, but not necessarily countable abelian group with operation $(g,h)\mapsto gh$ and identity~$1$. 
Since $G$ is abelian, it is amenable (see~\cite{Pier1984}), hence $G$ has a \Folner\ sequence, i.e.,  a sequence $\{\G_n\mid n\in\NN\}$ of nonempty finite subsets of~$G$ such that, for every $g\in G$, $\#(g\G_n \triangle \G_n)/\#\G_n \to 0$ as $n\to\infty$.
(Here, $A\triangle B$ denotes the symmetric difference of the sets~$A,B$.)
In what follows, $G$ and one specific \Folner\ sequence $\{\G_n\}$ will be fixed.

We will be dealing with structures of the form
\[
(\R,\cH,\mathcal{B},\{T_g\}_{g\in \G},\{a_j\}_{j\in J}),
\]
where
\begin{itemize}
\item $\R$ is the field of real numbers. 
\item $\cH$ is a real Hilbert space with inner product $\left\langle \cdot,\cdot \right\rangle$ and norm $\Ltwon{x} = \sqrt{\left\langle x,x \right\rangle}$.
\item $\mathcal{B}$ is the real $C^{*}$-algebra of bounded operators~$\cH\to\cH$ endowed with the adjunction $B\mapsto B^{*}$, the evaluation map $\ev : \mathcal{B}\times\cH \to \cH: (B,x)\mapsto B(x)$, the composition map $(B_1,B_2)\mapsto B_1\circ B_2$ (the algebra product), and the operator norm $\Ltwon{B} = \sup_{x\in\cH} \Ltwon{B(x)}$.
\item $\{T_g \mid g\in G\}$ is a fixed unitary representation of~$G$ in~$\cH$: $T_1 = I$, $T_gT_h = T_{gh} = T_{hg} = T_hT_g$ and $T_g T^{*}_g = I = T^{*}_gT_g$ for all $g,h\in G$. 
\item For each $j\in J$, $a_j$ is an element of one of the sets $\mathbb{R}$, or $\cH$, or $\mathcal B$. 
\end{itemize}
Each of $\R, \cH, \cB$ will be called a \emph{sort} of the structure:
$\R$ is the \emph{real} sort,
$\cH$ the \emph{Hilbert} (or \emph{vector}) sort, and $\mathcal{B}$ the \emph{operator} sort. 
An \emph{element} of the preceding structure is an element of $\R\cup\cH\cup\cB$. 
The distinguished elements $a_j$ are called the \emph{constants} of the structure.

If $\G$ is any nonempty finite subset of~$G$, define the averaging operator
\begin{equation*}
  \AV_{\G} = \frac{1}{\#\G}\sum_{g\in \G} T_g \in \mathcal{B},
\end{equation*}
and let $\AV_n \coloneqq \AV_{\G_n}$.

\subsection{The Mean Ergodic Theorem}
\label{sec:MET}

\begin{theorem}\label{thm:MET}
  Let $G$ be an abelian group with a fixed \Folner\ sequence~$\{\G_n \mid n\in\NN\}$, $\cH$ a Hilbert space, and $\{T_g \mid g\in G\}$ a unitary representation of~$G$ on~$\cH$.

Let $\gH = \ker(T_g-I) = \{u\in \cH \mid T_g(u) = u\}$ and $\GH = \bigcap_{g\in G}\gH = \{u\in \cH \mid \text{$T_g(u)=u$ for all $g\in G$}\}$.
Let $\pi : \cH \to \GH$ be the orthogonal projection.

Then we have $\AV_n(x) \to \pi(x)$ as $n\to\infty$ for all $x\in \cH$.
\end{theorem}

\section{A logical language, discrete truth, and approximate truth}
\label{sec:language}

Fix a structure
\[
\cS=(\R,\cH,\mathcal{B},\{T_g\}_{g\in \G},\{a_i\}_{i\in I})
\]
as given in Section~\ref{sec:notation}.
We need a syntactic language $L$ to express logical statements about the structure~$\cS$ by regarding $\cS$ as an $L$-structure. 
Variables are needed: $\mathtt{t}_1, \mathtt{t}_2, \dots$ for the real sort, $\mathtt{x}_1, \mathtt{x}_2, \dots$ for the vector sort,  
$\mathtt{B}_1, \mathtt{B}_2, \dots$ for the operator sort.
We use the function symbol $\Ltwon{\cdot}$ for the Hilbert and operator norms (abusing notation, also for the real absolute value when needed), and $\left\langle x,y \right\rangle$ as an abbreviation for $\frac{1}{4}(\Ltwon{x+y}-\Ltwon{x-y})$, for any terms $x,y$ of the vector sort.
Our language $L$ also includes function symbols $+,\cdot,{}^*,\circ,\mathtt{ev},\mathtt{T}_g$, respectively, for the operations of addition (of reals, vector and operators), scalar multiplication, adjunction, composition, the evaluation~$\ev:\cB\times\cH\to\cH$, and the unitary transformations~$T_g$, $g\in G$. 
Finally, $L$ must include a constant symbol $\mathtt{a}_i$  for each constant element $a_i$; as is the case with variables, each constant symbol must be associated with a specific sort.

For syntactic simplicity, we will rarely if ever use the symbol for evaluation explicitly. 
Thus, we regard $\mathtt{B}(\mathtt{x})$ as an alias for $\mathtt{ev}(\mathtt{B},\mathtt{x})$.

If $C$ is a set of constant symbols each of which comes associated with one of the three given sorts, we shall denote by $L[C]$ the language that results from expanding $L$ with the constants in $C$. 
If $\cS$ is a structure, $L[\cS]$ will denote the expansion of $L$ that results from adding a constant symbol (of the adequate sort) for each element of~$\cS$. 

Now we define the class of $L$-formulas. 
The terms of $L$ are obtained applying the usual rules, that is, starting with variables and constant symbols along with function symbols in a manner consistent with semantic interpretations of $L$-formulas in~$\cS$ (or some other structure of the same kind).
The atomic formulas of $L$ are all the inequalities of the forms $s \le t$ (or $t \ge s$ if preferred) for any real terms $s,t$.
General formulas are obtained inductively by using connectives and quantifiers. 
However there are two restrictions: the only connectives allowed are the positive boolean connectives $\land$ (conjunction) and $\lor$ (disjunction), and in place of the traditional first-order quantifiers $\exists,\forall$, we have the \emph{bounded quantifiers} i.e.,  for each $r\in\R$, $r>0$ we have the bounded existential quantifier $\exists^r$ and the bounded universal quantifier $\forall^r$. 
The resulting class of formulas is called the class of \emph{positive bounded $L$-formulas}. 

The satisfaction relation $\cS\models\varphi[a_1,\dots,a_n]$ where $\cS$ is an $L$-structure (of the kind defined in this section) and $\varphi$ is positive bounded $L$-formula, is defined  as in traditional first-order logic, by giving the symbols of the signature and the connectives $\land,\lor$ their natural interpretation in $\cS$, but allowing the quantifiers  $\forall^r,\exists^r$ to range only over elements of $\cS$ of norm at most $r$, that is:
\begin{itemize}
\item $\cS\models \exists^r \mathtt{b}\, \varphi$ if and only if $\cS\models\varphi[b,a_1,\dots,a_n]$ for some element $b$ of $\cS$ of norm at most $r$ belonging to the same sort as the variable~$\mathtt{b}$,
\item $\cS\models \forall^r \mathtt{b}\, \varphi$ if and only if $\cS\models\varphi[b,a_1,\dots,a_n]$ for every element $b$ of $\cS$ of norm at most $r$ belonging to the same sort as the variable~$\mathtt{b}$.

\end{itemize}

If $t$ and $s$ are terms, we write 
$t = s$ as a purely syntactic equivalent for $-s \le t \wedge t\le s$.

We stress the fact that the language $L$ does not include a symbol for equality nor, more importantly, a symbol for negation. 
In particular, there is no general manner to formulate a statement such as ``$P$ implies $Q$'' (i.e., ``$Q$ or not $P$'') in the language~$L$.

We now define a relation of \emph{approximation} between formulas $\varphi,\varphi'$ in the language~$L$, denoted $\varphi'<\varphi'$ (\emph{$\varphi'$ approximates $\varphi$}, or \emph{$\varphi$ is approximated by~$\varphi'$}), as follows:
\begin{itemize}
\item $x \ge y$ is approximated by $x+\varepsilon \ge y$ for whatever real $\varepsilon>0$;
\item $x \le y$ is approximated by $x \le y+\varepsilon$ for whatever real $\varepsilon>0$;
\item $\varphi\wedge\psi$ is approximated by $\varphi'\wedge\psi'$ whenever $\varphi'<\varphi$ and $\psi'<\psi$;
\item $\varphi\vee\psi$ is approximated by $\varphi'\vee\psi'$ whenever $\varphi'<\varphi$ and $\psi'<\psi$;
\item $\exists^r \mathtt{a} \,\varphi$ is approximated by $\exists^s \mathtt{a}\,\varphi'$ whenever $s>r$ and $\varphi'<\varphi$;
\item $\forall^r \mathtt{a} \,\varphi$ is approximated by $\forall^s \mathtt{a}\,\varphi'$ whenever $s<r$ and $\varphi'<\varphi$.
\end{itemize}
It is immediate that $<$ is a strict partial order.

A weaker notion of truth of sentences in the (same) language~$L$ is obtained by defining $\cS \appmod \varphi$ to mean that $\cS\models \varphi'$ whenever $\varphi'<\varphi$. 
If  $\varphi$ is a sentence (i.e., a formula without free variables) and $\cS\models\varphi$ we say that $\cS$ satisfies $\varphi$ \emph{exactly} and if $\cS\appmod\phi$ we say that $\cS$ satisfies $\varphi$ \emph{approximately}.

The \emph{complete theory} (or simply ``theory'') of an $L$-structure $\cS$ is the set of all $L$-sentences approximately satisfied by $\cS$.

It is clear from the definition of the approximation relation that $\cS\models\varphi$ implies $\cS\appmod\varphi$.
The converse is true if the sentence $\varphi$ includes only universal quantifiers (in particular, if $\varphi$ is quantifier-free), but not for general sentences~$\varphi$.

\section{Types}
\label{sec:types}

The \emph{weak negation} $\Neg(\varphi)$ of any formula $\varphi$ in $L$ is defined recursively as follows:
\begin{itemize}
\item The weak negation of $x \le y$ is $x \ge y$;
\item The weak negation of $x \ge y$ is $x \le y$;
\item The weak negation of $\varphi\wedge\psi$ is $\Neg(\varphi)\vee\Neg(\psi)$;
\item The weak negation of $\varphi\vee\psi$ is $\Neg(\varphi)\wedge\Neg(\psi)$;
\item The weak negation of $\forall^r x\,\varphi$ is $\exists^r x\,\Neg(\varphi)$;
\item The weak negation of $\exists^r x\,\varphi$ is $\forall^r x\,\Neg(\varphi)$.
\end{itemize}
Weak negation transposes the relation of weak approximation: $\varphi'<\varphi$ if and only if $\Neg(\varphi)<\Neg(\varphi')$.
Clearly, $\Neg(\Neg(\varphi))$ is $\varphi$.
 
It is possible for a formula and its weak negation to both be exactly satisfied: $0_{\cH}$ satisfies $\Ltwon{\mathtt{0}_{\cH}}\ge \mathtt{0}$ and $\Ltwon{\mathtt{0}_{\cH}} \le \mathtt{0}$.
On the other hand, it is easy to see that a formula and the weak negation of any of its approximations cannot both be satisfied.

Recall that $L[\cS]$ includes one constant symbol $\mathtt{a}$ denoting each element $a$ of every sort of~$\cS$.
Then $\bS$ is an approximate elementary extension of $\cS$ (denoted $\cS \precsim \bS$) provided
\begin{equation*}
\cS\appmod \varphi\qquad\text{if and only if}\qquad \bS \appmod \varphi
\end{equation*}
for all $L[\cS]$-sentences~$\varphi$. 
This means that approximate truth of $L[\cS]$-formulas ``transfers'' between~$\cS$ and $\bS$.
By interpreting the constant symbols of $L[\cS]$ in the overstructure~$\bS$ one obtains injections of the vector and operator sorts of~$\cS$ into those of~$\bS$ that may be regarded as set-theoretical inclusions. 

We regard the formula $\iota : \Ltwon{\mathtt{b}-\mathtt{a}}\le 0$ as defining equality.
Explicitly, two elements $u,v$ of a sort are the same (for all analytic purposes) provided $\iota[u,v]$ is satisfied.
(For a quantifier-free formula such as $\iota$, approximate satisfaction is equivalent to exact satisfaction.)

We shall use the notation $\overline{a}$ to denote an arbitrary (possibly empty) list $a_1, \dots, a_k$ of elements of any sorts of some $L$-structure. 
We call $\overline{a}$ an \emph{$X$-tuple} if $a_1,\dots,a_k$ all belong to some subset~$X$ (of the $L$-structure).
Similarly,  $\overline{\mathtt{z}} = \mathtt{z}_1,\dots,\mathtt{z}_k$ denotes a list of variables (bound to whatever sorts).
We say that $\overline{a}$ and $\overline{\mathtt{z}}$ are \emph{sort-compatible} if $a_i$ belongs to the sort to which the variable~$z_i$ is bound for $i=1,\dots,k$.

Let $X$ be any subset of an $L$-structure~$\cS$.
Denote by $L[X]$ the language $L$ expanded with distinct constant symbols naming all elements of~$X$. A \emph{type over~$X$}, relative to (the theory of)~$\cS$, is a set $\Xi$ of $L[X]$-formulas satisfying the following properties:
\begin{enumerate}
\item\emph{$\Xi$ is over~$X$:} There shall be a fixed-length list $\overline{\mathtt{z}}$ of variables such that $\Xi$ contains only $L[X]$-formulas of the form
 $\varphi(\overline{\mathtt{z}})$ (i.e., the free variables of any and all $\varphi\in\Xi$ belong to the fixed list $\overline{\mathtt{z}}$), in which case we may write $\Xi(\overline{\mathtt{z}})$ instead of~$\Xi$ for emphasis;
\item\emph{$\Xi$ is norm-bounded:} For each variable $\mathtt{w}$ in $\overline{\mathtt{z}}$ there exists $r\in\R$ such that the formula $\Ltwon{\mathtt{w}} \le \mathtt{r}$ belongs to~$\Xi$;
\item\emph{$\Xi$ is relative to (the theory of)~$\cS$:} Every finite subset of $\Xi_{+}$ is (exactly) satisfied in~$\cS$;
\item\emph{$\Xi$ is complete:} For every $\varphi$ we have: either $\varphi\in\Xi$, or else $\Neg(\varphi')\in\Xi$ for some $\varphi'<\varphi$.
\end{enumerate}

For any $\cS$-tuple $\overline{b}$ of elements in any sorts, the \emph{type of $\overline{b}$ over~$X$} (relative to~$\cS$) is the set
\begin{equation*}
  \begin{split}
    \tp_X(\overline{b}) &=
    \{\varphi(\overline{\mathtt{z}}) \in L[X] :
    \cS \appmod \varphi[\overline{b}]\} \\
    &= \{\psi(\overline{\mathtt{z}},\overline{\mathtt{a}}) \in L[X] :
    \text{$\psi(\overline{\mathtt{z}},\overline{\mathtt{w}})\in L$, $\overline{a}$ is an $X$-tuple and $\cS\appmod
      \psi[\overline{b},\overline{a}]$\}}     
  \end{split}
\end{equation*}
(where $\overline{\mathtt{a}}$ is the tuple of constant symbols naming the elements of the $X$-tuple $\overline{a}$).
Note that $\tp_X(\overline{b})$ actually depends on (the $\approx$-theory) of~$\cS$, although our notation hides this implicit dependence.
It is easy to show that $\tp_X(\overline{b})$ is a type over~$X$ (relative to~$\cS$) per the definition above. 
By  the compactness theorem for the logic of approximate satisfaction (see~\cite{Henson-Iovino:2002}), the converse is true, that is, every type $\Xi$ over~$X$ relative to~$\cS$ is the type of some $\overline{b}$ relative to a suitable $\approx$-extension $\fS\succsim \cS$. 
We say that $\overline{b}$ \emph{realizes} the type~$\Xi$ in~$\fS$.

\section{The logic topology on types}
\label{sec:topology}

Given any set~$\Xi$ of formulas, let $\Xi_+$ be the set consisting of every formula $\varphi'$ approximating some formula $\varphi$ of~$\Xi$.

Fix a tuple $\overline{\mathtt{z}}$ of variables, an $L$-structure~$\cS$, and a subset $R\subset\cS$.
The set $\fT = \fT_X(\overline{\mathtt{z}})$ of types $\Xi(\overline{\mathtt{z}})$ over~$X$ (relative to~$\cS$) admits the following topology (subsequently called the \emph{logic topology} on~$\fT$).
Given any $L[X]$-formula $\varphi$, let $[\varphi]$ be the set of types containing~$\varphi$.
A basis of neighborhoods of a given type~$\Xi$ consists of the sets $[\varphi]$ with $\varphi\in\Xi_+$ (i.e., $\varphi<\psi$ for some $\psi\in\Xi$).
It is easy to show that the logic topology is Hausdorff.

Define the norm $\Ltwon{\Xi}$ of a type $\Xi(\overline{\mathtt{z}})$ as the infimum of all real numbers $r\ge 0$ such that $\Xi$ includes the formulas $\Ltwon{\mathtt{a}} \le \mathtt{r}$ for each variable $\mathtt{a}$ of $\overline{\mathtt{z}}$.
By definition of type, $\Ltwon{\Xi}$ is finite.
The set $\fT^r$ of types of norm at most $r$ is compact in the logic topology~\cite{Henson-Iovino:2002,Iovino2014}.
Hence, any bounded set of types (relative to the norm $\Ltwon{\cdot}$) is relatively compact.

\section{Homogeneous extensions, saturated extensions, and compactness}
\label{sec:extensions-compactness}

Recall that the \emph{density} of a topological space $X$, denoted $\density(X)$, is the smallest cardinality of a dense subset of $X$. 
The density of a structure $\cS$, denoted $\density(\cS)$, is the sum of the densities of the sorts of $\cS$. 
The cardinality of a structure $\cS$, denoted $\card(\cS)$, is the sum of the cardinalities of the sorts of $\cS$. 

Let $\kappa$ be an infinite cardinal.

\begin{itemize}
\item
An $L$-structure $\cS$ is said to be \emph{$\kappa^+$-saturated} if any type (relative to the theory of~$\cS$) over a subset $X\subset\cS$ with $\density(X)\le\kappa$ is realized in~$\cS$.
\item
An $L$-structure $\cS$ is said to be \emph{strongly $\kappa^+$-homogeneous} if the following condition holds: if $C$ is a set of constants not in $L$ of cardinality at most $\kappa$ and if  $\{a_c\}_{c\in C}, \{b_c\}_{c\in C}$ are interpretations of the constants of $C$ in $\cS$ such that the structures $(\cS,a_c)_{c\in C}, (\cS,b_c)_{c\in C}$ approximately satisfy the same positive bounded $L[C]$-sentences, then there is an automorphism $f$ of the structure $\cS$ such that $f(a_c)=b_c$, for every $c\in C$.
\end{itemize}

For arbitrarily large $\kappa$, every structure $\cS$ has an approximate elementary extension $\tS$ such that $\tS$ is  $\kappa^+$-saturated and strongly  $\card(\cS)^+$-homogeneous; moreover, $\tS$ can be taken to be ultrapower of $\cS$. 
See~\cite{Henson-Iovino:2002}, Theorem 12.2 and Corollary 12.3.

Henceforth we fix an approximate elementary extension $\tS\succsim \cS$ such that $\tS$ is $\card(\cS)^+$-saturated and strongly $\density(\cS)^+$-homogeneous. 
Since $\cS\succsim\tS$, types over~$\cS$ relative to~$\cS$ are the same as those relative to~$\tS$.
The saturation condition means that any type (relative to the theory of~$\cS$) over a subset $X\subset\tS$ with $\density(X)\le\density(\cS)$ is realized in~$\tS$ itself. 
In particular, all possible types $\Xi_{\cS}(\overline{\mathtt{z}})$ over $\cS$ (for any and all variable tuples~$\overline{\mathtt{z}}$) are realized in~$\tS$.
The strong  $\card(\cS)^+$-homogeneity implies that for any two $\tS$-tuples $\overline{a}$, $\overline{b}$ with $\tp_{\cS}(\overline{a}) = \tp_{\cS}(\overline{b})$ there exists an automorphism of $\tS$ leaving $\cS$ fixed (pointwise).

The Hilbert and operator sorts of $\tS$ will be denoted $\tH$, $\tB$, respectively.
By a simple argument using saturation, $\tH$ is a Hilbert space; moreover, $\tB$ may be identified with a $C^*$-subalgebra of the set of all bounded endomorphisms of~$\tH$.
(We remark that one cannot expect to obtain all Hilbert-space endomorphisms of~$\tH$ through this identification. 
In the terminology of classical nonstandard analysis, operators so obtained are called ``inner'', the rest are ``outer''.)

Fix any variable tuple~$\overline{\mathtt{z}}$.
We shall presently exploit the compactness of $\fT^r = \fT^r_{\cS}(\overline{\mathtt{z}})$ (types over~$\cS$, of norm at most~$r$, relative to the theory of~$\cS$ or any $\approx$-elementary extension thereof).

Recall that $\beta\mathbb{N}\setminus \mathbb{N}$ denotes the set 
\begin{equation*}
\{\,\U\subset\NN \mid \text{$\U$ is a nonprincipal ultrafilter on~$\NN$}\,\}.
\end{equation*}
Given any sequence $\{\Xi_n \mid n\in\NN\}\subset\fT^r$ and any $\U\in\beta\mathbb{N}\setminus \mathbb{N}$, there exists a (unique) ultralimit type $\Xi^{\U} = \Ulim_n \Xi_n$ (in the logic topology).
By saturation, every type over~$\cS$ is realized in~$\tS$, so there exists some $\cS$-tuple $\overline{a}$ such that $\Xi^{\U} = \tp_{\cS}(\overline{a})$.
The tuple $\overline{a}$ need not be unique.
However, it is easy to see that $\overline{a}$ is unique (in the metric sense) if $\overline{a}$ is an $\cS$-tuple; indeed, in this case, $\tp_{\cS}(\overline{a})$ includes each of the formulas $\Ltwon{\mathtt{w} - \mathtt{b}}\le 0$ for each variable~$\mathtt{w}$ in $\overline{\mathtt{z}}$ and element $b$ in~$\overline{a}$.

\section{Proof of the mean ergodic theorem}
\label{sec:proof}
This section is devoted to the proof of theorem~\ref{thm:MET}.

Throughout the section, $\cS=(\R,\cH,\mathcal{B},\{T_g\}_{g\in \G}),$ is a fixed structure, and $\tS = (\R,\tH, \tB, \{\widetilde{T}_{g\in \G}\})$ will be a fixed a $\card(\cS)^+$-saturated, strongly $\density(\cS)^+$-homogeneous approximate elementary extension of $\cS$. 
All the types mentioned will be relative to the theory of $\cS$. 
We shall assume that our basic language $L$ includes a constant symbol  for each element of~$\cS$. 

Whenever $\Phi,\Psi$ are types over some common subset~$X$ of $\tS$, we shall write $\Phi\models\Psi$ to mean that every realization of $\Phi$ in $\tS$ is a realization of~$\Psi$.

Recall that $\mathtt{t}_1, \mathtt{t}_2, \dots$, $\mathtt{x}_1, \mathtt{x}_2, \dots$, and  $\mathtt{B}_1, \mathtt{B}_2, \dots$ are the variables for the real sort, the vector sort, and the operator sort, respectively. 
We shall define types in these variables. 
(For readability, we omit the subindices when the context allows it.)

Let $\AV_n = \frac{1}{C_n}\sum_{g\in\G_n} T_g$ be the averaging operator on~$\tS$, i.e., $\AV_n$ is the $\tS$-interpretation of the term $\mathtt{AV}_n \coloneqq \mathtt{C}_n^{-1}\big(\mathtt{T}_{g_1} + \ldots + \mathtt{T}_{g_{C_n}})$, where $g_1,\dots,g_{C_n}$ are the $C_n$ distinct elements of~$\G_n$.
Note that $\Ltwon{\AV_n} \le 1$ for all $n\in\NN$ since $\Ltwon{T_g} = 1$ for all $g\in G$.

For any $\U\in\beta\mathbb{N}\setminus \mathbb{N}$, and any $x\in\tH$, define the following types  over~$\cS$:
\begin{itemize}
\item $\Sigma^{\U}_x(\mathtt{B},\mathtt{x}) \coloneqq \Ulim_n\tp_{\cS}(\AV_n,x)$;
\item $\Upsilon^{\U}_x(\mathtt{y}) \coloneqq \Ulim_n\tp_{\cS}(\AV_n(x))$; and
\item $\Xi^{\U}(\mathtt{B}) \coloneqq \Ulim_n\tp_{\cS}(\AV_n)$.
\end{itemize}
(The above ultralimits exist by compacity of $\fT^r$ with $r = \max\{\Ltwon{x},1\}$.)

For any $n\in\NN$, we have $\tp_{\cS}(x) \models \tp_{\cS}(\AV_n,x)$.
In fact, for any $L[\cS]$-formula $\xi(\mathtt{B},\mathtt{x})$ we have $\xi(\mathtt{B},\mathtt{x})\in\tp_{\cS}(\AV_n,x)$ precisely when $\xi(\mathtt{AV}_n,\mathtt{x})\in\tp_{\cS}(x)$.
Next, $\tp_{\cS}(\AV_n,x)\models\tp_{\cS}(\AV_n(x))$ since $\xi(\mathtt{y}) \in \tp_{\cS}(\AV_n(x))$ precisely when $\xi(\mathtt{B}(\mathtt{x})) \in \tp_{\cS}(\mathtt{B},\mathtt{x})$. 
Note that $\tp_{\cS}(\AV_n)$ is determined completely by the theory of~$\cS$ alone.

Thus, for any fixed choice of~$\U\in\beta\mathbb{N}\setminus \mathbb{N}$, it is clear that $\tp_{\cS}(x) \models \Sigma^{\U}_x \models \Upsilon^{\U}_x$. 

For every real $\varepsilon\ge 0$, let $\varphi^{\varepsilon}(\mathtt{B},\mathtt{x})$ be the formula ``$\Ltwon{\mathtt{B}(\mathtt{x})}\le\varepsilon$'' and $\psi^{\varepsilon}(\mathtt{y})$ the formula ``$\Ltwon{\mathtt{y}}\le\varepsilon$''.
Call $x$ a \emph{$\U$-null-ergodic} element (or just \emph{$\U$-null}), if $\psi^0 \in \Upsilon^{\U}_x$.
Call $x$ \emph{null-ergodic} (or just \emph{null}) if $x$ is $\U$-null for all~$\U\in\beta\mathbb{N}\setminus \mathbb{N}$.
(Note that ``$\psi^0\in\Upsilon^{\U}_x$'' is equivalent to the statement that $\Upsilon^{\U}_x$ is $\tp_{\cS}(0_{\cH})$, the type over~$\cS$ of the zero vector.)

Since $\tp_{\cS}(x)\models \Upsilon^{\U}_x$ for all~$\U\in\beta\mathbb{N}\setminus \mathbb{N}$, the property ``$x$ is $\U$-null'' (resp., ``$x$ is null'') depends only on $\U$ and $\tp_{\cS}(x)$ (resp., only on $\tp_{\cS}(x)$).

\begin{lemma}\label{lem:nullity}
An element $x$ is null if and only if
  \begin{equation*}
    \text{$\Ltwon{\AV_n(x)} \to 0$\quad as $n\to\infty$.}
  \end{equation*}
\end{lemma}
\begin{proof}
  By properties of ultrafilters we have, for each $\varepsilon > 0$, the statement:
  \begin{equation*}
    \text{``For all $\U\in\beta\mathbb{N}\setminus \mathbb{N}$ : $\{n\in\NN \mid \Ltwon{\AV_n(x)} \le \varepsilon \} \in \U$"}
  \end{equation*}
is equivalent to the statement: ``$\Ltwon{\AV_n(x)}\le\varepsilon$ for all sufficiently large~$n\in\NN$''.
Thus, the formulas $\psi^{\varepsilon}(\mathtt{y})$ for all $\varepsilon > 0$ belong to~$\Upsilon_x$, and so $x$ is null, precisely when $\lim_{n\to\infty} \Ltwon{\AV_n(x)} = 0$ in the usual sense.
\end{proof}

Let $\tA_{\U}$ be the set of realizations of $\Xi^{\U}$ in~$\tS$ (i.e., in~$\tB$), and let $\tA = \bigcup_{\U\in\beta\mathbb{N}\setminus \mathbb{N}} \tA_{\U}$. 
Let $\tHn_{\U}$ be the set of $\U$-null elements in~$\tH$, 
and $\tHn = \bigcap_{\U\in\beta\mathbb{N}\setminus \mathbb{N}}\tHn_{\U}$ the set of null elements of~$\tH$.
After Tao, we use the adjective \emph{``$\U$-pseudorandom''} to refer to any element $x\in\tH$ with $\Ltwon{A(x)} = 0$ for all $A\in\tA_{\U}$.
Similarly, ``\emph{pseudorandom}'' means ``$\U$-pseudorandom for all $\U\in\beta\mathbb{N}\setminus \mathbb{N}$''. 
Let $\tHr_{\U}$ be the set of $\U$-pseudorandom elements of~$\tH$.
Since $\tHr_{\U} = \bigcap_{A\in\tA_{\U}}\ker(A)$, $\tHr_{\U}$ is a closed subspace of~$\tH$.

\begin{lemma}
  \label{lem:r-implies-n}
For each ultrafilter $\U\in\beta\mathbb{N}\setminus \mathbb{N}$, $\tHr_{\U} \subset \tHn_{\U}$. 
(Every $\U$-pseudorandom element is $\U$-null.)
\end{lemma}
\begin{proof}
$\U$ is fixed throughout the proof.
For now, let $x$ be any element of $\tH$. 
$\Sigma^{\U}_x$ is realized in~$\tS$, hence there exist $A\in\tA_{\U}$ and $y\in\tH$ with $\tp_{\cS}(y) = \tp_{\cS}(x)$ such that $(A,y)$ realizes~$\Sigma^{\U}_x$.
By homogeneity, there exists an automorphism $f : \tS\to\tS$ fixing~$\cS$ such that $f(y) = x$.
Let $A' = f(A)$.
Clearly, $(A',x)$ realizes~$\Sigma^{\U}_x$ (because $f$ preserves satisfaction of $L$-formulas as it fixes~$\cS$).
In particular, $A'\in\tA_{\U}$.

Since $\Sigma^{\U}_x\models\Upsilon_{\U}$, we have
$\psi^0(\mathtt{y}) \in \Upsilon^{\U}_x$ if and only if $\varphi^0(\mathtt{B},\mathtt{x}) = \psi^0(\mathtt{B}(\mathtt{x})) \in \Sigma^{\U}_x$. 

Assume now that $x$ is $\U$-pseudorandom.
For some $A'\in\tA_{\U}$, $(A',x)$ realizes $\Sigma^{\U}_x$; moreover, $\Ltwon{A'(x)} = 0$ since $x$ is $\U$-pseudorandom.
Therefore, $\varphi^0(\mathtt{B},\mathtt{x}) \in \tp_{\cS}(A',x)  = \Sigma^{\U}_x$, so $\psi^0(\mathtt{y}) \in \Upsilon^{\U}_x$ and $x$ is $\U$-null.
\end{proof}

Let $\tHr = \bigcap_{\U\in\beta\mathbb{N}\setminus \mathbb{N}} \tHr_{\U}$ be the set of all pseudorandom elements of~$\tH$.
Clearly, $\tHr$ is a closed subspace of~$\tH$.

\begin{corollary}\label{cor:r-implies-n}
We have $\tHr\subset\tHn$.  
Every element of $\tHr$ is null-ergodic:
\begin{equation*}
  x\in\tHr\qquad\text{implies}\qquad \lim_{n\to\infty}\Ltwon{\AV_n(x)} = 0.
\end{equation*}
\end{corollary}
\begin{proof}
  Immediate consequence of lemmas~\ref{lem:nullity} and~\ref{lem:r-implies-n}.
\end{proof}

\begin{lemma}\label{lem:G-invar} For any ultrafilter $\U\in\beta\mathbb{N}\setminus \mathbb{N}$, every $g\in G$, every $n\in\NN$, and any realization $A$ of $\Xi^{\U}$, we have
\begin{enumerate}
\item $\Ltwon{T_g\circ A - A} = 0$, and
\item $\Ltwon{A\circ T_g-A} = 0$.
\end{enumerate}
\end{lemma}
\begin{proof}
Let $\sigma^{\varepsilon}_g(\mathtt{B})$ be the formula $\Ltwon{\mathtt{T}_g\circ \mathtt{B} - \mathtt{B}} \le \varepsilon$.

With $C_n = \#\G_n$ (where $\#P$ denotes the cardinality of the set~$P$), let $e(g,n) \coloneqq \#(\G_n\triangle  g\G_n)/C_n$. 
(Here, $P\triangle Q = (P\setminus Q)\cup(Q\setminus P)$ is the symmetric difference of the sets $P, Q$.) By definition of \Folner\ sequence we have $e(g,n)\to 0$ as $n\to\infty$. 

Clearly, $\Ltwon{A}\le 1$ inasmuch as $\Xi^{\U}$ contains the formula ``$\Ltwon{\mathtt{B}}\le 1$'' ($\AV_n$ is a convex combination of the norm-$1$ (unitary) operators $T_g$, so $\Ltwon{\AV_n} \le 1$ holds for all~$n\in\NN$).
Now,
  \begin{equation*}
    \begin{split}
      \Ltwon{T_g\circ\AV_n - \AV_n} &=
      \Ltwon{\frac{1}{C_n}\sum_{h\in\G_n}T_g\circ T_h -
        \frac{1}{C_n}\sum_{h\in\G_n} T_h}
      \\
      &= \Ltwon{\frac{1}{C_n}\sum_{h\in g\G_n}T_h -
        \frac{1}{C_n}\sum_{h\in\G_n} T_h}
      \qquad\text{since $T_{gh} = T_g\circ T_h$}\\
      &= \Ltwon{\frac{1}{C_n}\sum_{h\in g\G_n\setminus\G_n} T_h - \frac{1}{C_n}\sum_{h\in \G_n\setminus g\G_n} T_h} \\
      & \le C_n^{-1}\sum_{h\in g\G_n\setminus\G_n} \Ltwon{T_h} +
      C_n^{-1}\sum_{h\in \G_n\setminus g\G_n} \Ltwon{T_h}
      \\
      &= C_n^{-1}\#(\G_n\triangle g\G_n)\cdot 1 = e(g,n).
    \end{split}
\end{equation*}
Therefore, $\Ltwon{T_g\circ\AV_n - \AV_n} \le \varepsilon$ (i.e., $\varphi_g^{\varepsilon}(\mathtt{B}) \in \tp_{\cS}(\AV_n)$) for all sufficiently large~$n$.
Thus, $\varphi_g^{\varepsilon}(\mathtt{B}) \in \Xi^{\U}$ (for any~$\U$).
Since $A$ realizes~$\Xi^{\U}$,  the (quantifier-free) statement~$\varphi^0_g(A)$ is approximately, hence exactly, satisfied, so $\Ltwon{T_g\circ A-A} = 0$, proving~(1). 

Statement~(2) follows from the fact that the operators $\{T_g\mid g\in G\}$ commute in (any elementary $\approx$-extension of)~$\cS$, i.e., $\Ltwon{T_g\circ T_h - T_h\circ T_g} = 0$, whence $\Ltwon{\AV_n\circ T_g-\AV_n} = \Ltwon{T_g\circ\AV_n - \AV_n}$ follows by the triangle inequality, and the argument above shows $\Ltwon{A\circ T_g - A} = 0$. 
\end{proof}

\begin{corollary}\label{cor:G-invar}
For any ultrafilter $\U\in\beta\mathbb{N}\setminus \mathbb{N}$, every $g\in G$, every $n\in\NN$, and any realization $A$ of $\Xi^{\U}$, we have:
\begin{enumerate}
\item $\Ltwon{T_g\circ A^{*} - A^{*}} = 0$, and
\item $\Ltwon{A^{*}\circ T_g - A^{*}} = 0$.
\end{enumerate}
\end{corollary}
\begin{proof}
  Both statements follow from lemma~\ref{lem:G-invar} by taking adjoints and using the fact that $\Ltwon{T^{*}_g-T_{g^{-1}}} = 0$ (since the $T_g$ are unitary by hypothesis). 
\end{proof}

After Tao, let $\tHs_{\U} = (\tHr_{\U})^{\perp}$ (the set of \emph{$\U$-structured} elements of~$\tHs$) be the orthogonal complement of~$\tHr_{\U}$, and let $\tHs = (\tHr)^{\perp}$ (the set of (all) \emph{structured} elements).
For each $A\in\tA$, $(\ker A)^{\perp} = \img(A^{*})$ ($A$ is a bounded operator on the Hilbert space~$\tH$).
Since $\tHr = \bigcap_{A\in\cA}\ker A$, the set $\tHs$ of structured elements is the closure of the set 
\begin{equation*}
  \tHsf = \left\{\overline{A^{*}}\cdot \overline{x} \coloneqq \sum_{i=1}^n A_i^{*}(x_i)\right\}
\end{equation*}
of finite linear combinations of elements $A_i^{*}(x_i)$ in the images $\img(A_i^{*})$, for some tuples $\overline{A} = A_1,\dots,A_n$ in $\cA$, and $\overline{x} = x_1,\dots,x_n$ in~$\tH$, both of some (non-fixed) finite length~$n$.

Recall that $\GtH$ is the set of $G$-fixed elements of~$\tH$.

\begin{lemma}\label{lem:G-fixed-disjoint-null}
 $x\in\GtH\cap\tHn$ implies~$\Ltwon{x} = 0$, i.e., $\GtH$ and $\tHn$ are independent subspaces of~$\tH$.
\end{lemma}
\begin{proof}
Observe that elements of $\GtH$ are $\AV_n$-fixed for all~$n\in\NN$, by linearity. 
Let $x\in\GtH\cap\tHn$. 
We have,
  \begin{equation*}
    \begin{split}
      \Ltwon{x} &= \Ltwon{\AV_n(x)} \qquad\text{for all $n$, since
        $x\in\GtH$}\\
   &=  \lim_{n\to\infty} \Ltwon{\AV_n(x)} = 0 \qquad\text{by lemma~\ref{lem:nullity}, since $x\in\tHn$.}\qedhere
    \end{split}
  \end{equation*}
\end{proof}

\begin{lemma}\label{lem:s-implies-G}
$\tHs\subset\GtH$, i.e., $\Ltwon{T_g(x)-x} = 0$ for all $x\in\tHs$ and all $g\in G$.
\end{lemma}
\begin{proof}
Let $x\in\tHs$, $g\in G$.
Let $\varepsilon > 0$ be arbitrary.
Choose $n$-tuples $\overline{A}$, $\overline{z}$ such that $y \coloneqq \overline{A^{*}}\cdot \overline{z} \in \tHsf$, and $w \coloneqq x-y$ satisfies $\Ltwon{w} \le \varepsilon$.
By (2)~of corollary~\ref{cor:G-invar} and the inequality $\Ltwon{B(v)} \le \Ltwon{B}\Ltwon{v}$, we have $\Ltwon{T_g(A_i^{*}(z_i)) - A_i^{*}(z_i)} = 0$ for each~$i$.
Thus, $\Ltwon{T_g(y)-y} = 0$ follows by linearity (and the triangle inequality), so we have $\Ltwon{T_g(x)-x} \le \Ltwon{T_g(y)-y} + \Ltwon{T_g(w)-w} \le 0+2\Ltwon{w} < 2\varepsilon$ for all $n\in\NN$.
Since $\varepsilon>0$ was arbitrary, we have $\Ltwon{T_g(x)-x} = 0$ for all $g\in G$.
Thus, $x\in\GtH$.
\end{proof}

\begin{proof}[Proof of the Mean Ergodic Theorem]
  First, we show that the limit of any sequence $\{\AV_n(x)\mid n\in\NN\}$ (for any~$x\in\tH$) exists.

Write $x = \xr + \xs$ with $\xr\in\tHr$ and $\xs\in\tHs$.

We have $\lim_{n\to\infty}\Ltwon{\AV_n(\xr)} = 0$, by corollary~\ref{cor:r-implies-n}.

Now, from lemma~\ref{lem:s-implies-G} and linearity, $\Ltwon{\AV_n(\xs)-\xs} = 0$ for all $n\in\NN$.
Therefore, $\lim_{n\to\infty}\AV_n(\xs) = \xs$.
Define $\Pi : \tH\to\tH$ by $\Pi(x) = \xs$ when $x = \xr + \xs$ as above.
($\Pi$ is the orthogonal projection onto~$\tHs$.) 
Thus, $\Pi(x) = \lim_{n\to\infty}\AV_n(x)$.
Note that $\Pi$ is ``external'' in the sense that $\Pi$ may not be realized by any $B\in\tB$, but $\Pi$ is nonetheless a \emph{bona fide} bounded operator on the Hilbert space~$\tH$ (strictly speaking, the genuine Hilbert space is the set $\tH$ modulo the equivalence relation $\Ltwon{\mathtt{v}-\mathtt{u}} = 0$).

Next, we show that $\Pi(x)$ is, in fact, the projection on the space~$\GtH$ of $G$-fixed points of~$\tH$.
Clearly, it suffices to show that~$\tHs = \GtH$.
Let $\Pi$ be the orthogonal projection~$\tH\to\GtH$.
We have:
\begin{enumerate}
\item $\tHr \subset \tHn$, by corollary~\ref{cor:r-implies-n};
\item $\tHs \subset \GtH = \img\Pi$, by lemma~\ref{lem:s-implies-G};
\item $\tHn$ and $\GtH$ are independent, by lemma~\ref{lem:G-fixed-disjoint-null}.
\item $\tHr$ and $\tHs$ generate $\tH$, since $\tHs$ is the (orthogonal) complement of~$\tHr$ (and~$\tH$ is Hilbert, hence complete).
\end{enumerate}
We conclude that $\tHs = \GtH$.
Therefore, $\Pi(x)$ is the orthogonal projection of~$x$ on~$\GtH$.

It remains to show that $\Pi\restriction_{\cH}$ agrees (point-by-point) with the orthogonal projection operator $\pi : \cH\to\GH$.
The key property is that both spaces $\GtH = \img\Pi$ and $(\GtH)^{\perp} = \ker\Pi$ are (semi)definable over~$\cS$ in a technical sense we will not presently discuss, but that is presently captured by the fact that the properties ``$x\in\GtH$'', ``$x\in\big(\GtH\big)^{\perp}$'' depend only on the type of~$x$ over~$\cS$.

Clearly, $x\in\GtH$ if and only if $\tp_{\cS}(x)$ includes all the formulas ``$\Ltwon{\mathtt{T}_g(\mathtt{x})-\mathtt{x}} = 0$'' for $g\in G$. 
Moreover, these are exactly the formulas ensuring that an element $x\in\cH$ belongs to~$\GH$; in particular, $\GH = \GtH\cap\cH$. 
Further, since $\img(\tT_{g^{-1}}-I) = \img (\tT^{*}_g-I) = \ker(\tT_g-I)^{\perp}$, we have $x\in\big(\GtH\big)^{\perp}$ if and only if for every $\varepsilon > 0$ there exists $r\in\R$, a finite tuple~$\overline{g} = g_1,\dots,g_k$ in~$G$, and a formula
\begin{equation*}
  \varphi_{\overline{g}}(\mathtt{x}) : \exists^r\overline{\mathtt{y}}(\Ltwon{\mathtt{x} 
+ \mathtt{y}_1 + \dots + \mathtt{y}_k
- \mathtt{T}_{g_1}(\mathtt{y}_1) - \dots - \mathtt{T}_{g_k}(\mathtt{y}_k)} \le \varepsilon)
\end{equation*}
in $\tp_{\cS}(x)$.
These, however, are exactly the conditions that ensure that an element $x\in\cH$ belongs to~$\GH^{\perp}$.
Thus, $\GH^{\perp} = \big(\GtH\big)^{\perp} \cap \cH$.
We conclude that $\pi$ is the restriction of~$\Pi$ to~$\cH$, finishing the proof of theorem~\ref{thm:MET} on~$\cH$.
\end{proof}

\section{Concluding remarks}
\label{sec:conclusion}

It appears that \emph{a priori} knowledge of (expected) properties of $\Pi$ in the context of von Neumann's result enables the existence of slick short proofs such as Riesz's.
A ``nonstandard analytic'' approach (i.e., the framework of types in Banach structures in the context of this manuscript) seems most valuable when there is insufficient \emph{a priori} knowledge about the nature of any potential limits.

In this manuscript, $\Pi$ is external, although its values $\Pi(x)$ can fortunately be described in the language~$L$ through the type of the point~$x$.
In a more general context such as that of ergodic averages \emph{\`a la} Walsh one cannot hope, in general, to give a very explicit description of the operator $\Pi = \lim_{n\to\infty}\AV_n(\cdot)$.

As a final remark, Zorin-Kranich~\cite{ZorinKr2011} has generalized Walsh's result to polynomial nilpotent actions of amenable groups in precisely the manner that Wiener's theorem (theorem~\ref{thm:MET} here) extends the discrete-time formulation of von Neumann's mean ergodic theorem.

\bibliography{bibdatabase,iovino}
\bibliographystyle{halpha}

\end{document}